\documentclass[a4,12pt]{amsart}
\usepackage{amssymb,amstext,amsmath,amscd,amsthm,amsfonts,enumerate,graphicx,latexsym}

\newcommand{\rt}{\rightarrow}
\newcommand{\lrt}{\longrightarrow}

\def\spec{\operatorname{\mathsf{Spec}}}

\newcommand{\Z}{\mathbb{Z} }

\newcommand{\C}{\mathcal{C} }

\newcommand{\m}{\mathfrak{m}}

\newcommand{\CI}{\mathcal{I} }

\newcommand{\CT}{\mathcal{T} }

\newcommand{\Y}{\mathcal{Y} }

\newcommand{\p}{\mathfrak{p} }

\newcommand{\Hom}{{\rm{Hom}}}
\newcommand{\End}{{\rm{End}}}

\newcommand{\depth}{\operatorname{\mathsf{depth}}}
\def\nf{\operatorname{\mathsf{NF}}}

\def\sing{\operatorname{\mathsf{Sing}}}
\def\K{\operatorname{\mathsf{K}}}

\newcommand{\Ext}{{\rm{Ext}}}
\newcommand{\ca} {\mathsf{ca}}

\def\db{\operatorname{\mathsf{D^b}}}
\def\thick{\operatorname{\mathsf{thick}}}

\def\fl{\operatorname{\mathsf{fl}}}
\def\radius{\operatorname{\mathsf{radius}}}
\def\ds{\operatorname{\mathsf{D_{sg}}}}
\def\Hom{\operatorname{\mathsf{Hom}}}
\def\lcm{\operatorname{\mathsf{\underline{CM}}}}
\def\cm{\operatorname{\mathsf{CM}}}
\def\xx{\text{\boldmath $x$}}
\def\yy{\text{\boldmath $y$}}
\def\Ext{\operatorname{\mathsf{Ext}}}
\def\H{\operatorname{\mathsf{H}}}
\def\X{\mathcal{X}}
\def\ext{\operatorname{\mathsf{ext}}}
\def\add{\operatorname{\mathsf{add}}}
\def\mod{\operatorname{\mathsf{mod}}}
\def\syz{\mathsf{\Omega}}
\def\supp{\operatorname{\mathsf{Supp}}}
\def\v{\operatorname{\mathsf{V}}}
\def\q{\mathfrak q}
\def\ann{\operatorname{\mathsf{ann}}}

\newcommand{\der}{{\rm{der}}}

\newtheorem{theorem}{Theorem}[section]

\newtheorem{cor}[theorem]{Corollary}

\newtheorem{lem}[theorem]{Lemma}
\newtheorem{prop}[theorem]{Proposition}

\theoremstyle{definition}
\newtheorem{dfn}[theorem]{Definition}

\newtheorem{conv}[theorem]{Convention}

\newtheorem{rem}[theorem]{Remark}

\theoremstyle{plain}

\theoremstyle{definition}

\theoremstyle{remark}
\newtheorem*{ac}{Acknowledgments}

\numberwithin{equation}{section}

\begin{document}
\baselineskip=14.96pt

\title[Annihilation of cohomology, generation of modules and derived dimension]{Annihilation of cohomology, generation of modules and finiteness of derived dimension}

\author[Bahlekeh, Hakimian, Salarian, Takahashi]{Abdolnaser Bahlekeh, Ehsan Hakimian, Shokrollah Salarian, Ryo Takahashi}

\address{Department of Mathematics, Gonbad Kavous University, Postal Code:4971799151, Gonbad-e-Kavous, Iran and School of Mathematics, Institute for Research in Fundamental Science (IPM), P.O.Box: 19395-5746, Tehran, Iran}
\email{n.bahlekeh@gmail.com}

\address{Department of Mathematics, University of Isfahan, P.O.Box: 81746-73441, Isfahan,
 Iran}\email{hakimian@sci.ui.ac.ir}

\address{Department of Mathematics, University of Isfahan, P.O.Box: 81746-73441, Isfahan,
 Iran and School of Mathematics, Institute for Research in Fundamental Science (IPM), P.O.Box: 19395-5746, Tehran, Iran}
 \email{Salarian@ipm.ir}

\address{Graduate School of Mathematics, Nagoya University, Furocho, Chikusaku, Nagoya 464-8602, Japan}
\email{takahashi@math.nagoya-u.ac.jp}
\urladdr{http://www.math.nagoya-u.ac.jp/~takahashi/}



\thanks{The first and third authors were partly supported by grants from IPM (No. 93130055 and 93130218). The fourth author was partly supported by JSPS Grant-in-Aid for Scientific Research (C) 25400038}

\begin{abstract}
Let $(R,\m,k)$ be a commutative noetherian local ring of Krull dimension $d$.
We prove that the cohomology annihilator $\ca(R)$ of $R$ is $\m$-primary if and only if for some $n\ge0$ the $n$-th syzygies in $\mod R$ are constructed from syzygies of $k$ by taking direct sums/summands and a fixed number of extensions.
These conditions yield that $R$ is an isolated singularity such that the bounded derived category $\db(R)$ and the singularity category $\ds(R)$ have finite dimension, and the converse holds when $R$ is Gorenstein.
We also show that the modules locally free on the punctured spectrum are constructed from syzygies of finite length modules by taking direct sums/summands and $d$ extensions.
This result is exploited to investigate several ascent and descent problems between $R$ and its completion $\widehat R$.
\end{abstract}
\maketitle

\section{Introduction}
Let $R$ be a commutative noetherian local ring with maximal ideal $\m$, and let $\mod R$ denote the category of all finitely generated $R$-modules.
For any non-negative integer $n$, the elements of $R$ annihilating $\Ext_R^n(M, N)$, for all $M$ and $N$ in $\mod R$, form an ideal which, following \cite{ua}, we denote by $\ca^n(R)$.
It is easy to see that there is a tower of ideals
$$
\cdots\subseteq\ca^n(R)\subseteq\ca^{n+1}(R)\subseteq\cdots,
$$
so their union $\ca(R)$ is also an ideal of $R$, which is said to be {\em cohomology annihilator} of $R$.
As $R$ is noetherian, there exists an integer $s$ such that $\ca(R)=\ca^s(R)$.
Unless $R$ is regular, $\ca(R)$ is a proper ideal.

The notion of the cohomology annihilator was introduced and studied independently by Dieterich \cite{di} and Yoshino \cite{y} in connection with the Brauer--Thrall conjectures for maximal Cohen--Macaulay modules, where they proved that the cohomology annihilator of a $d$-dimensional Cohen-Macaulay complete local ring with perfect coefficient field is $\m$-primary, provided that $R$ is an isolated singularity.
Later on, Popescu and Roczen \cite{pr} removed the assumption being an isolated singularity from the result of Dieterich and Yoshino.
A nice theorem of Auslander \cite{A} states that every complete Cohen--Macaulay local ring of finite Cohen--Macaulay representation type is an isolated singularity.
In fact, he essentially proved that in this case, the cohomology annihilator $\ca(R)$ is $\m$-primary.
This result was extended by Leuschke and Wiegand \cite{LW} to the case where the ring is excellent, and by Huneke and Leuschke \cite{HL} to all Cohen--Macaulay local rings.
Recently, Iyengar and Takahashi \cite{ua} considered the cohomology annihilator of a noetherian ring that is finitely generated as a module over its center, that is, a noether algebra.

On the other hand, the notion of dimension for triangulated categories was introduced by Bondal--Van den Bergh and Rouquier in \cite{BV, R} and analogues for abelian categories by Dao--Takahashi \cite{radius, dim}.
These essentially indicate the number of {\em extensions} necessary to build all objects out of a single object.
Using the dimension of a bounded derived category, Rouquier presented the first example of an artin algebra of representation dimension greater than three \cite{R1}.
Since then, various studies concerning the finiteness of the dimension of a bounded derived category have been made; see Remark \ref{hist}.
Iyengar and Takahashi \cite{ua} investigated the relationship between the existence of non-trivial cohomology annihilators and the existence of strong generators for the category of finitely generated modules and its derived categories.

The main theme of this paper is to study cohomology annihilators of commutative noetherian rings and investigate their connections with the dimensions of a bounded derived category and a singularity category, as well as the radius of a subcategory of finitely generated modules.
It is shown that having a non-trivial cohomology annihilator guarantees the existence of a non-negative integer $n$ such that a given $R$-module $M$ is built out of the syzygies of an $R/{\ca(R)}$-module by taking $n$-extensions, up to finite direct sums and direct summands.
It also turns out that the subcategory of $\mod R$ consisting of modules whose nonfree loci are those prime ideals containing $\ca(R)$, is an extension closure of syzygies of $R/{\p}$, where $\p$ runs over the prime ideals containing $\ca(R)$.
Indeed, we shall prove a more general result in Theorem \ref{mthm}, which is the main result of this paper and all of the other results given in this paper are deduced from this.
As the first application of this result, we show that $\ca(R)$ is $\m$-primary if and only if for some integer $n\ge 0$, the subcategory consisting of $n$-th syzygies, $\syz^n(\mod R)$, is built out of syzygies of the residue field of $R$ in a finite number of extensions, direct sums and direct summands.
Moreover, these statements imply that the dimensions of the bounded derived category and the singularity category are finite.
Surprisingly, it is shown all of these statements are equivalent, assuming $R$ is Gorenstein; see Theorem \ref{5cond}.
It is well-known that the singular locus of $R$ is contained in the defining closed subset of the cohomology annihilator of $R$, $V(\ca(R))$.
In Theorem \ref{sing}, we investigate when this containment becomes an equality and when the cohomology annihilator of $R$ is non-trivial.

Another application of Theorem \ref{mthm} leads us to showing that the subcategory of $R$-modules that are locally free on the punctured spectrum of $R$ are built out of syzygies of modules of finite length by taking $d$ {\em extensions} in $\mod R$ up to finite direct sums and direct summands, where $d$ is the Krull dimension of $R$, and consequently, this subcategory is contained in the extension closure of syzygies of the residue field $k$; see Theorem \ref{tdthm}.
So, this result removes the assumption that $R$ is Cohen--Macaulay from \cite[Theorem 2.4]{stcm}.
Also, it is shown that, $\ca(R)$ is $\m$-primary if and only if the cohomology annihilator the ($\m$-adic) completion $\widehat R$ of $R$, $\ca(\widehat{R})$, is $\widehat{\m}$-primary, provided that $\widehat{R}$ is an isolated singularity; see Theorem \ref{cahat}.
The importance of this result comes from the fact that it has long been known that many nice properties of local rings need not be inherited by their completions and this has always regarded as a theoretical limitation.
It is also proved that finiteness and countability of the set of isomorphism classes of indecomposable maximal Cohen--Macaulay modules ascends and descends between a Henselian local ring $R$ and its completion, whenever $\widehat{R}$ is an isolated singularity; see Corollary \ref{fcmth}.
So, in this case Schreyer's conjecture \cite[Conjecture 7.3]{S} holds true.
The proof of this result enables us to show that, under the same assumption,
$R$ satisfies the Auslander--Reiten conjecture if and only its completion does so.

\section{Basic definitions}

This section is devoted to stating the definitions and basic properties
of notions which we will freely use in the later sections.
Let us start with our convention.
\begin{conv}
Throughout the paper, let $R$ be a commutative noetherian ring with identity.
We assume that all modules are finitely generated and that all subcategories are full and strict (i.e., closed under isomorphism).
\end{conv}

\begin{dfn}
(1) We denote by $\mod R$ the category of (finitely generated) $R$-modules.\\
(2) The {\em singular locus} of $R$, denoted by $\sing R$, is by definition the set of prime ideals $\p$ of $R$ such that $R_{\p}$ is not a regular local ring.\\
(3) A local ring $(R,\m)$ is called an {\em isolated singularity} if $R_\p$ is regular for all nonmaximal prime ideals $\p$, that is, $\sing R\subseteq \{\m\}$.
\end{dfn}

\begin{dfn}
Let $M$ be an $R$-module.\\
(1) The {\em nonfree locus} $\nf(M)$ of $M$ is defined as the set of prime ideals $\p$ of $R$ such that $M_\p$ is a nonfree $R_\p$-module.
It is well-known (and easy to see) that $\nf(M)$ is a closed subset of $\spec R$ in the Zariski topology.\\
(2) We say that $M$ is {\em locally free on the punctured spectrum of $R$} if $M_\p$ is a free $R_\p$-module for all nonmaximal prime ideals $\p$, namely, $\nf(M)\subseteq\{\m\}$.\\
(3) Let $\xx$ be a sequence of elements of $R$.
Then $\K(\xx, M)$ denotes the {\em Koszul complex} of $M$ with respect to $\xx$.
For each integer $i$ the $i$-th homology $\H_i(\xx, M):=\H_i(\K(\xx, M))$ is called the {\em $i$-th Koszul homology} of $M$ with respect to $\xx$.
The direct sum $\H(\xx, M):=\bigoplus_{i\in\Z}\H_i(\xx, M)$ is called the {\em Koszul homology} of $M$ with respect to $\xx$.
\end{dfn}

\begin{dfn}
Let $\X$ be a subcategory of $\mod R$.\\
(1) We say that $\X$ is a {\em resolving} subcategory of $\mod R$ if it contains projective modules and closed under direct summands, extensions and kernels of epimorphisms.
This notion has been introduced by Auslander and Bridger \cite{ab}.\\
(2) $\X$ is said to be a {\em Serre} subcategory if it is closed under submodules, quotient modules and extensions.
This is equivalent to saying that for each short exact sequence $0\to L\to M\to N\to 0$ in $\mod R$ the module $M$ is in $\X$ if and only if $L,N$ are in $\X$.\\
(3) We say that $\X$ is a {\em thick} subcategory of $\mod R$ if is closed under direct summands and short exact sequences.
The latter condition means that for each short exact sequence $0\to L\to M\to N\to 0$ in $\mod R$, if two of $L,M,N$ are in $\X$, then so is the third.\\
(4) For each integer $n\ge 0$, let $\syz^n\X$ denote the subcategory of $\mod R$ consisting of $n$-th syzygies of $R$-modules in $\X$, namely, those modules $M$ which admits an exact sequence $0 \to M \to P_{n-1} \to \cdots \to P_0 \to X \to 0$ with each $P_i$ free and $X\in\X$.\\
(5) The {\em additive closure} $\add\X$ (repsectively, {\em extension closure} $\ext\X$) of $\X$ is by definition the smallest subcategory of $\mod R$ containing $\X$ and closed under finite direct sums and direct summands (respectively, direct summands and extensions).
We denote by $\thick\X$ the {\em thick closure} of $\X$, namely, the smallest thick subcategory of $\mod R$ containing $\X$.\\
(6) Let $\CT$ be a triangulated category. A {\em thick subcategory} of $\CT$ is by definition a triangulated
subcategory of $\CT$ closed under direct summands. A {\em thick closure} of a subcategory $\Y$ of $\CT$, denoted $\thick\Y$,
is defined as a smallest thick subcategory of $\CT$ containing $\Y$. When $\Y$ consists of a single object $M$,
we denote it by $\thick M$.
\end{dfn}

\begin{rem}\label{rem}
If $n\ge1$, then the subcategory $\syz^n\X$ is closed under direct sums with free modules,
that is to say, if $M$ is an $R$-module in $\syz^n\X$, then so is $M\oplus F$ for all free $R$-modules $F$.
In fact, if there is an exact sequence $0 \to M \xrightarrow{f} P_{n-1} \xrightarrow{g} P_{n-2} \to \cdots \to P_0 \to X \to 0$ with each $P_i$ free and $X\in\X$, then the sequence $0 \to M\oplus F \xrightarrow{\left(\begin{smallmatrix}
f&0\\
0&1
\end{smallmatrix}\right)} P_{n-1}\oplus F \xrightarrow{\left(\begin{smallmatrix}
g&0
\end{smallmatrix}\right)} P_{n-2} \to \cdots \to P_0 \to X \to 0$ is exact.
\end{rem}

\begin{dfn}
Let $W$ be a subset of $\spec R$.\\
(1) The {\em dimension} of $W$, denoted $\dim W$, is defined as the supremum of $\dim R/{\p}$ where $\p$ runs through all prime ideals in $W$.
Hence $\dim W=-\infty$ if and only if $W$ is empty. \\
(2) We denote by $\supp^{-1}(W)$ (respectively, $\nf^{-1}(W)$) the subcategory of $\mod R$
consisting of $R$-modules whose supports (respectively, nonfree loci) are contained in $W$.
Note that $\supp^{-1}(W)$ (respectively, $\nf^{-1}(W)$) is a Serre (respectively, resolving) subcategory of $\mod R$.
\end{dfn}

\begin{dfn}
For subcategories $\X_1,\dots,\X_n$ of $\mod R$, we denote by $\bigoplus_{i=1}^n\X_i$
the subcategory of $\mod R$ consisting of modules of the form $\bigoplus_{i=1}^nX_i$ with $X_i\in\X_i$.
For a family ${\{\X_\lambda\}}_{\lambda\in\Lambda}$ of subcategories of $\mod R$ we
denote by $\bigcup_{\lambda\in\Lambda}\X_\lambda$ the subcategory of $\mod R$
consisting of modules $M$ such that $M\in\X_\lambda$ for some $\lambda\in\Lambda$.
\end{dfn}

\begin{dfn}\cite[Definition 3.2]{R}
Let $\CT$ be a triangulated category.\\
(1) Let $\CI,\CI_1,\CI_2$ be subcategories of $\CT$.
Let $\CI_1 \ast\CI_2$ denote the subcategory of $\CT$ consisting of objects $M$ such that there exists an exact triangle $I_1\to M\to I_2\to I_1[1]$ with $I_1\in\CI_1$ and $I_2\in\CI_2$.
Denote by $\langle\CI\rangle$ the smallest subcategory of $\CT$ contains $\CI$ and is closed under taking finite direct sums, direct summands and shifts.
Inductively one defines $\langle\CI\rangle_0=0$ and $\langle\CI\rangle_r=\langle\langle\CI\rangle_{r-1}\ast \langle\CI\rangle\rangle$ for $r\geq 1$.
For $\CI=\{M\}$ we simply denote it by $\langle M\rangle_r$.\\
(2) The dimension $\dim\CT$ of $\CT$ is defined as the infimum of integers $n\ge0$ such that $\CT=\langle M\rangle_{n+1}$ for some $M\in\CT$.
\end{dfn}

\begin{dfn}\cite[Definition 2.3]{radius}\label{raddef}
Let $\X,\Y,\C$ be subcategories of $\mod R$.\\
(1) We denote by $[\X]$ the smallest subcategory of $\mod R$ containing $\{R\}\cup\X$ that is closed under finite direct sums, direct summands and syzygies, i.e, $[\X]=\add\{\,\syz^iX\mid i\ge0,\,X\in\X\,\}$.
When $\X$ consists of a single object $M$, we simply denote it by $[M]$.\\
(2) We denote by $\X \circ \Y$ the subcategory of $\mod R$ consisting of objects $M$ which fits into an exact sequence $0\lrt X\lrt M\lrt Y\lrt 0$ in $\mod R$ with $X\in\X$ and $Y\in\Y$.
We set $\X \bullet \Y=[[\X]\circ [\Y]]$.\\
(3) The {\em ball of radius of $r$ centered at $\C$}, denoted by $[\C]_r$, is defined by $[\C]_0=0$ and $[\C]_r=[\C]_{r-1}\bullet[\C]=[[\C]_{r-1}\circ[\C]]$ for $r\ge2$.
For $\C=\{M\}$ we simply denote it by $[M]_r$.\\
(4) The {\em radius of $\X$}, denoted by $\radius\X$, is defined as the infimum of integers $n\ge 0$ such that there exists a ball of radius $n+1$ centered at a module containing $\X$.
\end{dfn}

\begin{dfn}\cite[Definition 5.1]{radius}
Let $\X,\Y$ be subcategories of $\mod R$.
Denote by $\X \circ \Y$ the subcategory of $\mod R$ consisting of objects $M$ which fits into an exact
sequence $0\lrt X\lrt M\lrt Y\lrt 0$ in $\mod R$ with $X\in\X$ and $Y\in\Y$.
Set $\X \bullet \Y=||\X|\circ |\Y||$, where $|\X|:=\add\X$.
Define $|\X|_r$ for each $r\ge0$ analogously to Definition \ref{raddef}(3).
This is nothing but the subcategory of $\mod R$ consisting of $R$-modules $M$ which is a direct summand of an $R$-module $N$ admitting a filtration $0=N_0\subseteq N_1\subseteq \cdots\subseteq N_r=N$ of $R$-submodules whose subquotients are in $\add\X$.
Note that $\ext\X={|\X|}_\infty:=\bigcup_{r\ge0}{|\X|}_r$.
\end{dfn}

\begin{dfn}
(1) We denote by $\db(R)$ the derived category of bounded complexes of $R$-modules, and identify $\mod R$ with the subcategory of $\db(R)$ consisting of complexes concentrated in degree zero.
Recall that the {\sl derived dimension} of $R$, denoted $\der.\dim R$, is defined to be the dimension of the triangulated category $\db(R)$.\\
(2) An $R$-complex is called {\it perfect} if it is bounded complex of projective $R$-modules.
The {\it singularity category} $\ds(R)$ of $R$, which is also called the {\em stable derived category} of $R$, is defined to be the Verdier quotient of $\db(R)$ by the perfect complexes.
For the definition of the Verdier quotient, we refer to \cite[Remark 2.1.9]{Ne}.
Whenever the singularity category $\ds(R)$ is discussed, we identify each object or subcategory of $\mod R$ with its image in $\ds(R)$ by the composition of the canonical functors $\mod R\rt\db(R)\rt\ds(R)$. The category $\ds(R)$ has been introduced and studied by Buchweitz \cite{bu} in connection with maximal Cohen--Macaulay modules over Gorenstein rings.
In recent years, it has been investigated by Orlov \cite{o1} in relation to the Homological Mirror Symmetry Conjecture.\\
(3) Let $R$ be a {\em (Iwanaga-)Gorenstein} ring, that is, $R$ has finite injective dimension as an $R$-module.
The {\em stable category} $\lcm(R)$ of maximal Cohen--Macaulay modules over $R$ is defined as follows.
The objects are the maximal Cohen--Macaulay $R$-modules, i.e., the $R$-modules $M$ with $\Ext^i_R(M, R)=0$ for all $i>0$.
The hom-set $\Hom_{\lcm(R)}(M, N)$ is the quotient of $\Hom_R(M, N)$ by the $R$-submodule consisting of homomorphisms $M\to N$ factoring through some projective $R$-modules.
It is known that $\lcm(R)$ is triangulated, and triangle equivalent to $\ds(R)$.
\end{dfn}

\begin{rem}\label{hist}
The importance of the notion of derived dimension was first recognized by Bondal and Van den Bergh \cite{BV} in relation to representability of functors.
In fact, they proved that smooth proper commutative/non-commutative varieties have finite derived dimension, which yields that every contravariant cohomological functor of finite type to vector spaces is representable.
Rouquier \cite{R} has proved that the derived dimension of coherent sheaves on a separated scheme of finite type over a field is finite.
Using the notion of derived dimension, Rouquier \cite{R1} constructed an example of artin algebra of representation dimension greater than three  for the first time.
Therefore he could solve a long standing open problem started with Auslander's work in his Queen Mary notes in 1971.
It is known that artinian algebras have finite derived dimension \cite{R}.
Christensen, Krause and Kussin \cite{C, KK} showed that rings of finite global dimension  have finite derived dimension, see also \cite[Proposition 8.3]{R}. More recently, Aihara and Takahashi \cite{AT} proved that derived dimension of a complete local ring with perfect coefficient field is finite. For small values of derived dimension, a number of definitive results have been obtained.
Rings of derived dimension zero have been classified.
It is shown, by Chen, Ye and Zhang \cite{CYZ} (see also \cite [Theorem 12.20]{Be}), that a finite dimensional algebra over an algebraically closed field has derived dimension zero if and only if it is an iterated tilted algebra of Dynkin type.
Recently, Iyengar and Takahashi \cite{ua} proved that an equicharacteristic excellent local ring and a commutative ring essentially of finite type over a field have finite derived dimension.
\end{rem}

\begin{dfn}\cite[Definition 2.1]{ua}
For each integer $n\ge0$, we set
$$
\textstyle\ca^n(R)=\ann_R\Ext_R^{\geqslant n}(\mod R,\mod R)=\bigcap_{M,N\in\mod R,\,i\ge n}\Ext_R^i(M,N),
$$
and call $\ca(R)=\bigcup_{n\ge0}\ca^n(R)$ the {\em cohomology annihilator} of $R$.
There is an ascending chain $0=\ca^0(R)\subseteq\ca^1(R)\subseteq\ca^2(R)\subseteq\cdots$ of ideals of $R$, and this stabilizes as $R$ is noetherian, namely, $\ca(R)=\ca^n(R)$ for $n\gg0$.
Note also that $\ca^n(R)=\ann_R\Ext_R^n(\mod R,\mod R)=\bigcap_{M,N\in\mod R}\ann_R\Ext_R^n(M,N)$.

As we have mentioned in the introduction, Iyengar and Takahashi \cite{ua}
mainly investigated the relation of the existence of non-trivial
cohomology annihilators and the existence of strong generators for the
category of finitely generated modules.
Recall that a finitely generated $R$-module $G$ is a {\em strong generator} for $\mod R$ if there exist non-negative integers $s$ and $n$ such that $\syz^s(\mod R)\subseteq|G|_n$.
\end{dfn}

\section{annihilation of cohomology and finiteness of derived dimension}
This section reveals a close link between the notion of cohomology
annihilator and finiteness of the derived dimension as well as the dimension of the singularity category. We start by stating and proving the most general
structure theorem; actually, all of the other
results of this paper are deduced from this.
\begin{theorem}\label{mthm}
Let $I$ be an ideal of $R$.
Let $\xx=x_1,\dots,x_n$ be a system of generators of $I$.
\begin{enumerate}[\rm(1)]
\item
Let $M$ be an $R$-module such that $I\Ext_R^1(M,\syz_RM)=0$.
Then there exists an $R/I$-module $L$ such that $M$ belongs to ${|\bigoplus_{i=0}^n\syz_R^iL|}_{n+1}$.
\item
One has
\begin{align*}
\nf^{-1}(\v(I))&=\textstyle\ext\left\{\syz_R^i(R/\p)\mid0\le i\le n,\,\p\in\v(I)\right\}\\
&=\textstyle{\left|\bigoplus_{i=0}^n\ext\syz_R^i\{R/\p\}_{\p\in\v(I)}\right|}_{n+1}\\
&=\textstyle{\left|\bigcup_{e>0,\,0\le i\le n}\syz_R^i(\mod R/I^{[e]})\right|}_{n+1}\\
&=\textstyle\bigcup_{e>0}\bigcup_{X\in\,\bigoplus_{i=0}^n\syz_R^i(\mod R/I^{[e]})}{|X|}_{n+1},
\end{align*}
where $I^{[e]}$ is the ideal generated by $\xx^e=x_1^e,\dots,x_n^e$.
\end{enumerate}
\end{theorem}
\begin{proof}
(1) Set $H_i=\H_i(\xx,M)$ and $L=H_0\oplus\cdots\oplus H_n$.
Since $\xx$ annihilates each $H_i$, one can regard $L$ as a module over $R/I$.
Using \cite[Lemma 2.13]{ua} and \cite[Corollary 3.2(2)]{kos}, we find exact sequences
$$
0 \to H_i \to E_i \to \syz_RE_{i-1} \to 0\quad(1\le i\le n)
$$
of $R$-modules with $E_0=H_0$ such that $M$ is a direct summand of $E_n$.
Hence for each $1\le i\le n$ there is an exact sequence
$$
0 \to \syz_R^{n-i}H_i \to \syz_R^{n-i}E_i \to \syz_R^{n-i+1}E_{i-1} \to 0,
$$
and an inductive argument shows that $M$ is in ${|\bigoplus_{i=0}^n\syz_R^iL|}_{n+1}$.

(2) We begin with showing the inclusion
\begin{equation}\label{1}
\nf^{-1}(\v(I))\supseteq\ext\{\syz_R^i(R/\p)\mid0\le i\le n,\,\p\in\v(I)\}.
\end{equation}
Since $\nf^{-1}(\v(I))$ is resolving, it is enough to check that $\syz_R^i(R/\p)$
belongs to $\nf^{-1}(\v(I))$ for $0\le i\le n$ and $\p\in\v(I)$.
Let $\q$ be a prime ideal of $R$.
The $R_\q$-module $\syz_R^i(R/\p)_\q$ is stably isomorphic to $\syz_{R_\q}^i(R_\q/\p R_\q)$,
and hence if $\q$ does not contain $\p$, then it is $R_\q$-free.
Hence we have $\nf(\syz_R^i(R/\p))\subseteq\v(\p)\subseteq\v(I)$, which implies that $\syz_R^i(R/\p)$ is in $\nf^{-1}(\v(I))$.
Thus \eqref{1} holds.

It is easy to see that the following two inclusions hold.
\begin{equation}\label{1.5}
\textstyle\ext\{\syz_R^i(R/\p)\mid0\le i\le n,\,\p\in\v(I)\}\supseteq{\left|\bigoplus_{i=0}^n\ext\syz_R^i\{R/\p\}_{\p\in\v(I)}\right|}_{n+1},
\end{equation}
\begin{equation}\label{3}
\textstyle{\left|\bigcup_{e>0,\,0\le i\le n}\syz_R^i(\mod R/I^{[e]})\right|}_{n+1}\supseteq\bigcup_{e>0}\bigcup_{X\in\,\bigoplus_{i=0}^n\syz_R^i(\mod R/I^{[e]})}{|X|}_{n+1}.
\end{equation}

Next, let us prove the inclusion
\begin{equation}\label{2}
\textstyle{\left|\bigoplus_{i=0}^n\ext\syz_R^i\{R/\p\}_{\p\in\v(I)}\right|}_{n+1}\supseteq{\left|\bigcup_{e>0,\,0\le i\le n}\syz_R^i(\mod R/I^{[e]})\right|}_{n+1}.
\end{equation}
Fix integers $e>0$ and $0\le i\le n$.
Pick an $R$-module $M$ in $\mod R/I^{[e]}$.
Take a filtration
$$
0=M_0\subseteq M_1\subseteq\cdots\subseteq M_r=M
$$
of $R$-submodules such that for each $1\le j\le r$ one has $M_j/M_{j-1}\cong R/\p_j$ with $\p_j\in\supp M$.
As $I^{[e]}$ annihilates $M$, the support of $M$ is contained in $\v(I)$.
Hence $\p_j$ is in $\v(I)$.
Applying the syzygy functor $\syz_R^i$ to the above filtration, we observe that $\syz_R^iM$ belongs to $\ext\syz_R^i\{R/\p\}_{\p\in\v(I)}$.
Now \eqref{2} follows.

Finally, we show that
\begin{equation}\label{4}
\textstyle\bigcup_{e>0}\bigcup_{X\in\,\bigoplus_{i=0}^n\syz_R^i(\mod R/I^{[e]})}{|X|}_{n+1}\supseteq\nf^{-1}(\v(I)).
\end{equation}
Let $M$ be an $R$-module whose nonfree locus is contained in $\v(I)$.
It then follows from \cite[Lemma 3.4]{kos} that there exists an integer $e>0$
such that the sequence $\xx^e=x_1^e,\dots,x_n^e$ annihilates $\Ext_R^i(M,N)$ for all $i>0$ and all $N\in\mod R$.
By (1) there is an $R/I^{[e]}$-module $L$ such that $M$ belongs to ${|X|}_{n+1}$, where $X:=\bigoplus_{i=0}^n\syz_R^iL$.
This shows the inclusion \eqref{4}.

Combining \eqref{1}, \eqref{1.5}, \eqref{3}, \eqref{2} and \eqref{4} completes the proof of the theorem.
\end{proof}

The theorem below highlights the benefit of considering cohomology annihilators for
commutative rings. In fact, this result makes precise the close link between the notion of
cohomology annihilator and other well-studied notions such as derived dimension, singularity
dimension and strong generators for module categories.

\begin{theorem}\label{5cond}
Let $(R,\m,k)$ be a $d$-dimensional local ring.
Consider the following five conditions.
\begin{enumerate}[\rm(1)]
\item
$\ca(R)$ is $\m$-primary.
\item
$\syz^n(\mod R)\subseteq{|\bigoplus_{i=0}^d\syz^ik|}_r$ for some $n\ge0$ and $r\ge1$.
\item
$R$ is an isolated singularity, and $\syz^n(\mod R)$ has finite radius for some $n\ge0$.
\item
$R$ is an isolated singularity, and $\db(R)$ has finite dimension.
\item
$R$ is an isolated singularity, and $\ds(R)$ has finite dimension.
\end{enumerate}

Then the implications {\rm(1)} $\Leftrightarrow$ {\rm(2)} $\Leftrightarrow$ {\rm(3)} $\Rightarrow$ {\rm(4)} $\Rightarrow$ {\rm(5)} hold.
If $R$ is Gorenstein, then the five conditions are equivalent.
\end{theorem}

\begin{proof}
(1) $\Rightarrow$ (2):
There are integers $n\ge0$ and $t\ge1$ such that $\m^t\Ext_R^{>n}(\mod R,\mod R)=0$.
Take a parameter ideal $Q$ of $R$ contained in $\m^t$.
Then $Q\Ext_R^{>0}(\syz^n(\mod R),\mod R)=0$.
Let $s\ge1$ be the Loewy length of the artinian ring $R/Q$, i.e., the minimal integer $i$ with $\m^i(R/Q)=0$.
The first assertion of Theorem \ref{mthm} implies that for each $R$-module $M$ in $\syz^n(\mod R)$
there exists an $R/Q$-module $L$ such that $M$ is in ${|\bigoplus_{i=0}^d\syz^iL|}_{d+1}$.
Note that $L$ belongs to ${|k|}_s$ as an $R$-module.
Hence $M$ is in ${|\bigoplus_{i=0}^d\syz^ik|}_{s(d+1)}$.

(2) $\Rightarrow$ (3):
For each nonmaximal prime ideal $\p$ of $R$ we have
$\syz^n(\mod R_\p)\subseteq{|(\bigoplus_{i=0}^d\syz^ik)_\p|}_r=\add R_\p$, which shows that $R_\p$ is regular (of dimension at most $n$).
Hence $R$ is an isolated singularity.
It is obvious that $\syz^n(\mod R)$ has finite radius.

(3) $\Rightarrow$ (1):
We see from \cite[Theorem 4.3]{ua} that $\sing R=\v(\ca(R))$.
Since $R$ is an isolated singularity, the ideal $\ca(R)$ is $\m$-primary.

(2) $\Rightarrow$ (4):
Set $G=\bigoplus_{i=0}^d\syz^ik$ and pick any module $M$ in $\mod R$.
Condition (2) implies that $\syz^nM$ is in $|G|_r$.
It is easy to see that in $\db(R)$ the module $M$ belongs to $\langle G\oplus R\rangle_{r+n}$.
Hence $\db(R)=\langle G\oplus R\rangle_{2(r+n)}$ by \cite[Proposition 2.6]{AAITY}.

(4) $\Rightarrow$ (5):
The implication holds trivially.

Now, suppose that $R$ is Gorenstein, and let us show the implication (5) $\Rightarrow$ (1).
By \cite[Proposition 4.3]{BFK} (see also \cite[Lemma 5.3(2)]{dim}),
there exists an integer $t\ge1$ such that $\m^t\Hom_{\ds(R)}(X,Y)=0$ for all $X,Y\in\ds(R)$.
Since $R$ is Gorenstein, Theorem 4.4.1 of \cite{bu} yields that the singularity category $\ds(R)$ is equivalent to the stable
category $\lcm(R)$ of maximal Cohen--Macaulay $R$-modules as an $R$-linear triangulated category,
and hence $\m^t\Hom_{\lcm(R)}(M,N)=0$ for all $M,N\in\lcm(R)$.
For $R$-modules $A,B$ there are isomorphisms
$$
\Ext_R^{d+1}(A,B)\cong\Ext_R^1(\syz^dA,B)\cong\Ext_R^2(\syz^dA,\syz B)\cong\cdots\cong\Ext_R^{d+1}(\syz^dA,\syz^dB)
$$
of $R$-modules; the first isomorphism is clear, and the other isomorphisms follow from the
fact that $\syz^dA$ is a maximal Cohen--Macaulay module over the Gorenstein local ring $R$.
Since there is an isomorphism $\Ext_R^{d+1}(\syz^dA,\syz^dB)\cong\Hom_{\lcm(R)}(\syz^dA,\syz^{-d-1}\syz^dB)$, we observe $\m^t\Ext_R^{d+1}(\mod R,\mod R)=0$, which implies that $\m^t$ is
contained in $\ca^{d+1}(R)$, whence in $\ca(R)$.
\end{proof}

It is fairly easy to see that the singular locus of $R$ is contained in the defining
closed subset of the cohomology annihilator ideal of $R$, $V(\ca(R))$; see \cite[Lemma 2.9(2)]{ua}.
So it is natural to ask that: For which rings is there an equality $\sing R=V(\ca(R))$?
As an attempt in this direction, Iyengar and Takahashi have shown that
 for   a commutative noetherian ring  $R$ which is either a finitely generated
algebra over a filed or an equicharacteristic excellent local ring the equality holds; see \cite[Theorems 5.3, 5.4]{ua},
(see also \cite[Theorem 4.3]{ua}).
The result below, can be considered as a partial answer to the raised question.

\begin{theorem}\label{sing}
Let $(R, \m)$ be a $d$-dimensional Gorenstein local ring with $\dim\ds(R)<\infty$ (e.g., with finite derived  dimension).
Then $V(\ca(R))=\sing R$.
In particular, if $R$ is reduced, then the ideal $\ca(R)$ contains a nonzerodivisor.
\end{theorem}

\begin{proof}
If $R$ is reduced, then $\dim\sing R<d$.
The last assertion follows from this.
Let us show the first assertion.
Combining our assumption with \cite[Theorem 4.4.1]{bu}, we have $\dim\lcm(R)<\infty$.
Take an object $M\in\cm(R)$ such that $\lcm(R)={\langle M\rangle}_n$, for some integer $n>0$.
Set $I=\ann_{R}\Ext_{R}^1(M, \syz^1M)$, and apply \cite[Lemma 2.13]{ua} to conclude that $I\Ext_{R}^i(M, \cm(R))=0$ for all $i>0$.
It is seen that $I\Hom_{\lcm(R)}({\langle M\rangle},\lcm(R))=0$.

Pick any $X,Y\in\lcm(R)={\langle M\rangle}_n$.
Then there is an exact triangle $A \to B \to C \to A[1]$ in $\lcm(R)$ with $A\in{\langle M\rangle}_{n-1}$ and $C\in\langle M\rangle$ such that $X$ is a direct summand of $B$.
There is an exact sequence
$$
\Hom_{\lcm(R)}(C,Y)\to\Hom_{\lcm(R)}(B,Y)\to\Hom_{\lcm(R)}(A,Y).
$$
By the induction hypothesis, $I^{n-1}$ and $I$ annihilate $\Hom_{\lcm(R)}(C,Y)$ and $\Hom_{\lcm(R)}(A,Y)$ respectively.
Hence $I^n$ annihilates $\Hom_{\lcm(R)}(B,Y)$, and $\Hom_{\lcm(R)}(X,Y)$.
It follows that $I^n\Hom_{\lcm(R)}(\lcm(R),\lcm(R))=0$, which implies $I^n\Ext_R^{>0}(\cm(R),\cm(R))=0$.
We thus obtain $V(\ca(R))\subseteq V(I)$.

In view of \cite[Lemma 2.9(2)]{ua}, it suffices to show that $V(I)\subseteq\sing R$.
To see this, let $\p$ be a prime ideal which is not in $\sing R$.
Then $R_{\p}$ is a regular ring.
As $M$ is maximal Cohen--Macaulay, $M_{\p}$ is a free $R_{\p}$-module, implying that $\p$ does not belong to $V(I)$.
\end{proof}

\begin{rem}
A more general version than Theorem \ref{sing} (for not necessarily local irngs) appears in \cite{jc}.
\end{rem}

The next theorem, asserts that existence of an ideal of $R$ whose
defining closed subset covers the nonfree locus of $\syz^n(\mod R)$, for any $n\ge 1$,
guarantees the existence of a subcategory $\X$ of $\mod R$ such that $\syz^n(\mod R)$ is contained in
extension closure of syzygies of $\X$ as well as
in the thick closure of $\X\cup R$, ensuring that singular locus of $R$ has finite dimension.
We first state a lemma, which is essentially included in \cite[Lemma 3.4]{kos}.

\begin{lem}\label{nfis}
Let $I$ be an ideal of $R$.
Let $M$ be an $R$-module.
The following are equivalent for each integer $n\ge0$.
\begin{enumerate}[\rm(1)]
\item
The syzygy $\syz^nM$ is in $\nf^{-1}(\v(I))$.
\item
There exists an integer $t>0$ such that $I^t\Ext_R^{>n}(M,\mod R)=0$.
\end{enumerate}
\end{lem}

\begin{proof}
(1) $\Rightarrow$ (2):
It follows from \cite[Lemma 3.4]{kos} that $I^t\Ext_R^{>0}(\syz^nM,\mod R)=0$ for some $t>0$.
Hence $I^t\Ext_R^{>n}(M,\mod R)=0$.

(2) $\Rightarrow$ (1):
Setting $N=\syz^nM$, we have $I^t\Ext_R^1(N,\syz N)=0$.
Therefore $\v(I)$ contains the support of the $R$-module $\Ext_R^1(N,\syz N)$, which coincides with $\nf(N)$.
\end{proof}

For a prime ideal $\p$ of $R$ we denote by $\X_\p$ the subcategory of $\mod R_\p$ consisting of modules of the form $X_\p$ with $X\in\X$.

\begin{theorem}\label{1dim}
Let $R$ be a local ring, and let $r\ge0$ be an integer.
Consider the following four conditions.
\begin{enumerate}[\rm(1)]
\item
There exist an ideal $I$ of $R$ with $\dim R/I\le r$ and an integer $n\ge0$ such that for each $R$-module $M$ there is an integer $t>0$ with $I^t\Ext_R^{>n}(M,\mod R)=0$.
\item
There exist a subcategory $\X$ of $\mod R$ with $\dim X\le r$ for all $X\in\X$ and an integer $n\ge0$ such that $\syz^n(\mod R)\subseteq{\left|\bigoplus_{i=0}^n\ext\syz^i\X\right|}_{n+1}$.
\item
There exist a subcategory $\X$ of $\mod R$ with $\dim X\le r$ for all $X\in\X$ and an integer $n\ge0$ such that $\syz^n(\mod R)\subseteq\thick(\X\cup\{R\})$.
\item
One has $\dim\sing R\le r$.
\end{enumerate}
Then the implications {\rm(1)} $\Rightarrow$ {\rm(2)} $\Rightarrow$ {\rm(3)} $\Rightarrow$ {\rm(4)} hold.
If $\sing R$ is closed, the four conditions are equivalent.
\end{theorem}

\begin{proof}
(1) $\Rightarrow$ (2):
Lemma \ref{nfis} implies that $\syz^n(\mod R)$ is contained in $\nf^{-1}(\v(I))$.
Putting $\X=\{R/\p\}_{\p\in\v(I)}$, we see from the second assertion of Theorem \ref{mthm} that $\nf^{-1}(\v(I))={\left|\bigoplus_{i=0}^n\ext\syz^i\X\right|}_{n+1}$.
For each $\p\in\v(I)$ we have $\dim R/\p\le\dim R/I\le r$.

(2) $\Rightarrow$ (3):
It is straightforward that ${\left|\bigoplus_{i=0}^n\ext\syz^i\X\right|}_{n+1}$ is contained in $\thick(\X\cup\{R\})$.

(3) $\Rightarrow$ (4):
Take a prime ideal $\p$ of $R$ with $\dim R/\p>r$.
Localization at $\p$ shows that $\syz^n(\mod R_\p)$ is contained in $\thick(\X_\p\cup\{R_\p\})$,
which coincides with the subcategory of $\mod R_\p$ consisting of modules of finite projective dimension since $\X_\p=0$.
Therefore $R_\p$ is regular.
This shows that $\sing R$ has dimension at most $r$.

Now assume that $\sing R$ is closed, and let us show that (4) implies (1).
There is an ideal $I$ of $R$ with $\sing R=\v(I)$.
As $\dim\sing R\le r$, we have $\dim R/I\le r$.
Let $\p$ be a prime ideal in $\nf(\syz^d(\mod R))$, where $d=\dim R$.
Then $\syz^dM_\p$ is not $R_\p$-free for some $R$-module $M$.
Hence $R_\p$ is not regular, that is, $\p$ is in $\sing R$.
Thus $\syz^d(\mod R)$ is contained in $\nf^{-1}(\v(I))$, and Lemma \ref{nfis} completes the proof.
\end{proof}

\section{extension-closed subcategories and annihilation of cohomology}
In this section we will obtain a structure theorem on modules over a local ring $(R,\m)$ that
are locally free on the punctured spectrum.
Using this result, we study extension closures of syzygies of the
residue field and investigate the relationships between cohomology
annihilators of $R$ and its $\m$-adic completion.

We denote by $\fl R$ the subcategory of $\mod R$ consisting of $R$-modules of
finite length and by $\mod_0R$ the subcategory of $\mod R$ consisting of
$R$-modules that are locally free on the punctured spectrum of $R$.

The following result, which is a consequence of the second assertion of Theorem \ref{mthm},
concerns the structure of modules which are locally free on the punctured spectrum.
We will observe that this result plays a key role in the proofs of other results in this section.

\begin{theorem}\label{tdthm}
Let $(R,\m,k)$ be a local ring of dimension $d$.
Let $M$ be an $R$-module that is locally free on the punctured spectrum of $R$.
Then $M$ belongs to ${|\bigcup_{i=t}^d\syz^i\fl R|}_{d-t+1}$, where $t=\depth M$.
In particular, $M$ is in $\ext(\bigoplus_{i=t}^d\syz^ik)$.
\end{theorem}

\begin{proof}
Take any system of parameters $\xx=x_1,\dots,x_d$ of $R$.
Then we have $\nf(M)\subseteq\{\m\}\subseteq\v(\xx)$, so $M$ is in $\nf^{-1}(\v(\xx))$.
It follows from the second assertion of Theorem \ref{mthm} that
\begin{equation}\label{in}
M\in{\left|\bigcup_{i=0}^d\syz^i\fl R\right|}_{d+1}.
\end{equation}

We can choose a sequence $\yy=y_1,\dots,y_t$ of elements in $R$ that is
both an $M$-sequence and a subsystem of parameters of $R$.
As $\nf(M)\subseteq\{\m\}\subseteq\v(\yy)$, applying \cite[Lemma 3.4]{kos}
again, we obtain $\yy^e\Ext_R^{>0}(M,\mod R)=0$ for some $e>0$.
Replacing $\yy$ with $\yy^e$, we may assume $\yy\Ext_R^{>0}(M,\mod R)=0$.
By \cite[Corollary 3.2(1)]{kos} the module $M$ is isomorphic to a
direct summand of $\syz_R^t(M/\yy M)$.
In view of the containment \eqref{in} for the $R/\yy R$-module $M/\yy M$,
it is seen that $M/\yy M$ is in ${|\bigoplus_{i=0}^{d-t}\syz_{R/\yy R}^i\fl(R/\yy R)|}_{d-t+1}$,
 where $|\ |$ is taken in $\mod R/\yy R$.
Appling the $t$-th syzygy functor over $R$ yields
$$
M\in{\left|\bigcup_{i=0}^{d-t}\syz_R^t\syz_{R/\yy R}^i\fl(R/\yy R)\right|}_{d-t+1}.
$$

Let $L$ be a module in $\fl(R/\yy R)$, and take an exact sequence
$$
0 \to \syz_{R/\yy R}^iL \to P_{i-1} \to \cdots \to P_0 \to L \to 0
$$
of $R/\yy R$-modules with $P_0,\dots,P_{i-1}$ free.
Appling the $t$-th syzygy functor over $R$ to this, we get an exact sequence
$$
0 \to \syz_R^t\syz_{R/\yy R}^iL \to \syz_R^tP_{i-1} \to \cdots \to \syz_R^tP_0 \to \syz_R^tL \to 0.
$$
Note that $\syz_R^tP_0,\dots,\syz_R^tP_{i-1}$ are free $R$-modules.
It follows that $\syz_R^t\syz_{R/\yy R}^iL$ is isomorphic to $\syz_R^{t+i}L\oplus F$ for some free $R$-module $F$.
As $L$ is also of finite length over $R$, we obtain
$$
M\in{\left|\bigcup_{i=0}^{d-t}\syz_R^{t+i}\fl R\cup\{R\}\right|}_{d-t+1}={\left|\bigcup_{i=t}^d\syz_R^i\fl R\cup\{R\}\right|}_{d-t+1}.
$$

We claim that $R$ belongs to $\add(\syz^d\fl R)$.
Indeed, if $d=0$, then $R$ has finite length, and belongs to $\add(\syz^d\fl R)=\mod R$.
If $d\ge1$, then taking any $R$-module $K$ of finite length, we see from Remark \ref{rem} that $\syz^dK\oplus R\in\syz^d\fl R$, and hence $R\in\add(\syz^d\fl R)$.

Consequently, $M$ belongs to ${|\bigcup_{i=t}^d\syz_R^i\fl R|}_{d-t+1}$,
and the proof of the first assertion of the theorem is completed.
The second assertion follows from the first assertion and the fact that $\fl R=\ext(k)$.
\end{proof}

\begin{rem}
In the case where $R$ is Cohen--Macaulay, the last assertion in
Theorem \ref{tdthm} is nothing but \cite[Theorem 2.4]{stcm}.
Theorem \ref{tdthm} not only removes the Cohen--Macaulay assumption
from \cite[Theorem 2.4]{stcm} but also gives more precise structure of the module $M$.
\end{rem}

We record here an immediate consequence of Theorem \ref{tdthm}.

\begin{cor}\label{strmod0}
Let $(R,\m,k)$ be a local ring of dimension $d$.
Then
$$
\mod_0R={\left|\bigcup_{i=0}^d\syz^i\fl R\right|}_{d+1}=\ext\left(\bigoplus_{i=0}^d\syz^ik\right).
$$
\end{cor}

\begin{proof}
Theorem \ref{tdthm} implies $\mod_0R\subseteq{|\bigcup_{i=0}^d\syz^i\fl R|}_{d+1}\subseteq\ext(\bigoplus_{i=0}^d\syz^ik)$.
Since $\mod_0R$ contains the module $\bigoplus_{i=0}^d\syz^ik$ and is closed under direct summands and extensions, it also contains $\ext(\bigoplus_{i=0}^d\syz^ik)$.
\end{proof}

The following result is a consequence of Corollary \ref{strmod0}.
An analogous result is obtained in \cite[Theorem 3.2]{stcm} over a Cohen--Macaulay local ring.

\begin{cor}\label{comp0}
Let $R$ be a local ring.
For every $M\in\mod_0\widehat R$ there exists $N\in\mod_0R$ such that $M$ is a direct summand of $\widehat N$.
\end{cor}

\begin{proof}
Set $G=\bigoplus_{i=0}^d\syz^ik\oplus R$.
The module $M$ is in $\ext_{\widehat R}(\widehat G)$ by Corollary \ref{strmod0}.
Using \cite[Proposition 3.1]{stcm}, we observe that there is an
$R$-module $N$ in $\ext_R(G)$ such that $M$ is a direct
summand of $\widehat N$, and $\ext_R(G)=\mod_0R$ by Corollary \ref{strmod0} again.
\end{proof}
The result below, which compares cohomology annihilators of a local ring $R$
and its completion, is an application of Corollary \ref{comp0}.

\begin{theorem}\label{cahat}
Let $R$ be a $d$-dimensional local ring.
Let $n\ge0$ be an integer.
\begin{enumerate}[\rm(1)]
\item
One has $\ca^n(\widehat R)\cap R\subseteq\ca^n(R)$.
\item
Suppose that $\widehat R$ is an isolated singularity.
Then $\ca^n(R)\subseteq\ca^{n+d}(\widehat R)\cap R$.
Hence
$$
\ca(R)=\ca(\widehat R)\cap R.
$$
In particular, $\ca(R)$ is $\m$-primary if and only if $\ca(\widehat R)$ is $\widehat\m$-primary.
\end{enumerate}
\end{theorem}
\begin{proof}
(1) Let $a\in\ca^n(\widehat R)\cap R$.
Then $a$ is an element of $R$ with $a\Ext_{\widehat R}^n(\mod\widehat R,\mod\widehat R)=0$.
Let $M,N$ be $R$-modules.
The element $a$ annihilates $\Ext_{\widehat R}^n(\widehat M,\widehat N)$, which is isomorphic to the completion of $\Ext_R^n(M,N)$.
Since the canonical map $R\to\widehat R$ is faithfully flat (and hence pure), $a$ annihilates $\Ext_R^n(M,N)$.
Therefore $a\in\ca^n(R)$.

(2) There is nothing to show for $n=0$, so let us assume $n>0$.
Take an element $a\in\ca^n(R)$.
Let $X,Y$ be $\widehat R$-modules.
Since $\widehat R$ is an isolated singularity, it is seen that $\syz_{\widehat R}^dX,\syz_{\widehat R}^dY$ are in $\mod_0\widehat R$.
Corollary \ref{comp0} implies that there exist $R$-modules $M,N\in\mod_0R$ such that
$\syz_{\widehat R}^dX,\syz_{\widehat R}^dY$ are direct summands of $\widehat M,\widehat N$ respectively.
The element $a$ annihilates $\Ext_R^n(M,N)$, and completion shows that it also annihilates $\Ext_{\widehat R}^n(\widehat M,\widehat N)$.
Hence $a$ annihilates $\Ext_{\widehat R}^n(\syz_{\widehat R}^dX,\syz_{\widehat R}^dY)$, which implies
$$
a\Ext_{\widehat R}^{n+d}(X,\syz_{\widehat R}^dY)=0
$$
as $n>0$.
Letting $Y=\syz^nX$, we get $a\Ext_{\widehat R}^{n+d}(X,\syz_{\widehat R}^{n+d}X)=0$, and
$a\Ext_{\widehat R}^{\geqslant(n+d)}(X,\mod\widehat R)=0$ by \cite[Lemma 2.13]{ua}.
Thus we obtain $a\in\ca^{n+d}(\widehat R)$.

Now we have proved $\ca^n(R)\subseteq\ca^{n+d}(\widehat R)\cap R$.
Combining this with (1) gives rise to the equality $\ca(R)=\ca(\widehat R)\cap R$.
The last assertion is straightforward from this.
\end{proof}

\begin{rem}
As we have mentioned in the introduction, rings whose cohomology annihilators are $\m$-primary,
 can be viewed as a generalization of those of finite CM-type.
So, it seems interesting to ask whether this result remains true,
if one replaces the condition that $R$ has finite CM-type with the condition that the cohomology annihilator of $R$ is $\m$-primary.
Theorem \ref{cahat} answers this question positively without the Cohen--Macaulay assumption of the base ring.
\end{rem}

As a direct consequence of the above theorem in conjunction
with Theorem \ref{5cond}, we include the following result.

\begin{cor}Let $(R, \m)$ be a Gorenstein local ring such that $\widehat{R}$ is an isolated singularity.
Then $\db(R)$ has finite dimension if and only if so does $\db(\widehat{R})$.
\end{cor}
\begin{proof}Assume that $\db(R)$ has finite dimension. Since $\widehat{R}$
is an isolated singularity, it is an elementary fact that the same is true for $R$.
So Theorem \ref{5cond} forces $\ca(R)$ to be $\m$-primary and then,
one may apply the above theorem and conclude that $\ca(\widehat{R})$
is $\widehat{\m}$-primary, as well. Now, another use of Theorem \ref{5cond}
finishes the proof.
\end{proof}

In the Henselian case, we have the following stronger result than Corollary \ref{comp0}.

\begin{prop}\label{twice}
Let $R$ be a Henselian local ring.
\begin{enumerate}[\rm(1)]
\item
If $M$ be an indecomposable $R$-module, then $\widehat M$ is an indecomposable $\widehat R$-module.
\item
For each $M\in\mod_0\widehat R$ there exists $N\in\mod_0R$ such that $M\cong\widehat N$.
\end{enumerate}
\end{prop}

\begin{proof}
(1) This result is essentially proved in \cite[Proposition 3.1]{LW2}.
For the convenience of the reader, we give a proof.
Let $E=\End_{R}(M)$ be the endomorphism ring of $M$, and let $J$ be the Jacobson radical of $E$.
Since $E$ is a module-finite $R$-algebra, it follows from \cite[Lemma 1.7]{lw} that $\m E\subseteq J$.
Therefore $E/J$ is an $R$-module of finite length, and so we obtain isomorphisms $\widehat{E}/{\widehat{J}}\cong {{E}/J}\otimes_{R}\widehat{R}\cong E/J$.
As $R$ is Henselian, according to \cite[Theorem 1.8]{lw}, $E/J$ is a division ring and consequently $\widehat{E}/{\widehat{J}}$ is a division ring as well, meaning that $\widehat{J}$ is a maximal ideal of $\widehat{E}$.
On the other hand, as $E/{\m E}$ is artinian, $J/{\m E}$ is a nilpotent ideal of $E/{\m E}$ and so the isomorphisms $\widehat{E}/\m\widehat{E}\cong E/\m E$ and $\widehat{J}/{\m\widehat{E}}\cong J/{\m E}$ yield that $\widehat{J}/{\m\widehat{E}}$ is also a nilpotent ideal of $\widehat{E}/{\m\widehat{E}}$.
Hence $\widehat{J}/{\m\widehat{E}}$ is contained in the Jacobson radical of $\widehat{E}/\m\widehat{E}$, which is equal to $J'/{\m\widehat{E}}$ with $J'$ the Jacobson radical of $\widehat{E}$.
Thus we obtain $\widehat{J}=J'$, ensuring that $\widehat{E}$ is a local ring, whence $\widehat{M}$ is an indecomposable $\widehat{R}$-module.

(2) We may assume that $M$ is an indecomposable $R$-module.
Corollary \ref{comp0} implies that there exists a module $X\in\mod_0R$ such that $M$ is a direct summand of $\widehat X$.
Let $X=\bigoplus_{i=1}^nX_i$ be an indecomposable decomposition.
Note that $X_i$ belongs to $\mod_0R$ for all $i$.
Since $R$ is Henselian, assertion (1) guarantees that each $\widehat R$-module $\widehat{X_i}$ is indecomposable.
We have $\widehat X\cong\bigoplus_{i=1}^n\widehat{X_i}$, and as $\widehat R$ is Henselian, $M\cong\widehat{X_t}$ for some $1\le t\le n$ by the Krull--Schmidt theorem.
\end{proof}

Recall that a commutative neotherian ring $R$ is said to be of {\em finite} (respectively, {\em countable}) {\em Cohen--Macaulay representation type} (CM-type, for short), if there exist only finitely (respectively, countably) many isomorphism classes of indecomposable maximal Cohen--Macaulay $R$-modules.

It was conjectured by Schreyer \cite[Conjecture 7.3]{S} that, a local ring $R$ is of finite CM-type if and only if the ($\m$-adic) completion of $R$ is of finite CM-type.
Examples, discovered by Leuschke and Wiegand \cite[Examples 2.1, 2.2]{LW}, disproved this conjecture.
However, several classes of rings do satisfy Schreyer's conjecture.
Namely, the conjecture was answered affirmatively by Wiegand \cite[Theorem 2.9(4)]{W} in the case where $R$ is a Cohen--Macaulay ring whose completion is an isolated singularity, and by Leuschke and Wiegand \cite[Main Theorem]{LW} in the case where $R$ is a Cohen--Macaulay excellent ring.
We should also point out that Schreyer's conjecture holds true for all one dimensional local rings; see \cite[Theorem 2.9(2)]{W}.
The next result guarantees the validity of Schreyer's conjecture in the case where
 $R$ is a Henselian local ring whose completion $\widehat R$ is an isolated singularity.

\begin{cor}\label{fcmth}
Let $(R, \m)$ be a $d$-dimensional local ring such that $\widehat{R}$ is an isolated singularity.
Then the following statements hold true.
\begin{enumerate}[\rm(1)]
\item
Suppose that $R$ is Henselian.
Then $R$ has finite (respectively, countable) $CM$-type if and only if so does $\widehat{R}$.
\item
Let $n\ge0$ be an integer.
If $\add\syz^n(\mod R)$ contains only finitely (respectively, countably) many nonisomorphic indecomposable modules, then so does $\add\syz^{n+d}(\mod \widehat{R})$.
\end{enumerate}
\end{cor}
\begin{proof}
We only deal with the finite case; the statements for the countable case are similarly shown.

(1) The `if' part holds true regardless of the assumption that $\widehat{R}$ is an isolated singularity, thanks to the fact that $\widehat{R}$ is faithfully flat over $R$; see \cite[Corollary 1.6]{W}.

So we have only to prove the `only if' part.
Let $\{X_1,\dots,X_t\}$ be a complete list of representatives for the isomorphism classes of indecomposable maximal Cohen--Macaulay $R$-modules.
Since $\widehat{R}$ is an isolated singularity, we see from Proposition \ref{twice} that $\{\widehat{X_1},\dots,\widehat{X_t}\}$ is a complete list of representatives for the isomorphism classes of indecomposable maximal Cohen--Macaulay $\widehat{R}$-modules.

(2) Let $\{X_1,\dots,X_t\}$ be a complete list of representatives for the isomorphism classes of indecomposable $R$-modules in $\add\syz^n(\mod R)$.
Let $M$ be an indecomposable $\widehat R$-module in $\add\syz^{n+d}(\mod \widehat{R})$.
Then $M$ is a direct summand of $\syz_{\widehat R}^{n+d}N$ for some $\widehat R$-module $N$.
Since $\widehat R$ is an isolated singularity, $\syz_{\widehat R}^dN$ is locally free on the punctured spectrum of $\widehat R$.
Hence there exists an $R$-module $L$ such that $\syz_{\widehat R}^dN$ is a direct summand of $\widehat L$ by Corollary \ref{comp0}.
Therefore $M$ is a direct summand of $\syz^n_{\widehat R}\widehat L=\widehat{\syz_R^nL}$.
Setting $X=X_1\oplus\cdots\oplus X_t$, one has that $\syz_R^nL$ belongs to $\add(X)$, and $M$ is in $\add(\widehat X)$
Thus $\add\syz^{n+d}(\mod\widehat R)=\add(\widehat{X})$.
Since $\widehat R$ is Henselian, the assertion follows from this.
\end{proof}

\begin{rem}
The {\sl Auslander--Reiten conjecture \cite{ar}} claims that, over an artin algebra $\Lambda$, if $M$
is a finitely generated $\Lambda$-module such that $\Ext_{\Lambda}^{i>0}(M,M\oplus \Lambda)=0$,
then $M$ is projective. This long-standing conjecture, which is rooted in a conjecture of
Nakayama \cite{Na}, is known to be true
for several classes of algebras including, algebras of finite representation type
\cite{ar} and symmetric artin algebras with radical cube zero \cite{Ho}.
The Auslander--Reiten conjecture actually makes sense for any noetherian ring.
In particular, Auslander, Ding and Solberg in \cite{ads}, studied the following condition
on a commutative noetherian ring, not necessarily an artin algebra.

\begin{enumerate}
\item[\textsf{(ARC)}]
Let $R$ be a commutative noetherian ring and $M$ a finitely generated
$R$-module. If $\Ext_{R}^i(M, M\oplus R)=0$ for all $i>0$, then $M$ is projective.
\end{enumerate}

\noindent
There are already some results in the study of classes of commutative rings
satisfying \textsf{(ARC)}; see for instance \cite{a, bfs, ct, ch, hl}.
In the next result, by using the above corollary, we investigate ascent and descent of \textsf{(ARC)} between a local ring and its completion.
\end{rem}

\begin{cor}\label{arc}
Let $(R, \m)$ be a $d$-dimensional Henselian local ring such that $\widehat{R}$ is an isolated singularity.
Then $R$ satisfies \textsf {(ARC)} if and only if so does $\widehat{R}$.
\end{cor}
\begin{proof}
Since $\widehat{R}$ is faithfully flat over $R$, it is easy to verify that descent holds true
even without assuming that $R$ is Henselian or that $\widehat{R}$ is an isolated singularity.
So let us show the ascent.
To this end, assume that $M$ is an arbitrary $\widehat{R}$-module such that $\Ext_{\widehat{R}}^{>0}(M, M\oplus\widehat{R})=0$.
We then deduce $\Ext_{\widehat{R}}^{>0}(\syz^nM, \syz^nM\oplus\widehat{R})=0$ for each $n\ge0$.
Moreover, since $\Ext_{\widehat{R}}^{i>0}(M, \widehat{R})=0$, one concludes that $M$ is free if and only if so is $\syz^nM$.
Thus, replacing $M$ with $\syz^dM$, we may assume $M\in\mod_0{\widehat{R}}$, since $\widehat R$ is an isolated singularity.
By Proposition \ref{twice} we have $M\cong\widehat N$ for some $N\in\mod_0R$.
Hence $\Ext_{R}^i(N, N\oplus R)\otimes_{R}\widehat{R}\cong\Ext_{\widehat{R}}^i(\widehat{N}, \widehat{N}\oplus\widehat{R})\cong\Ext_{\widehat R}^i(M,M\oplus\widehat R)=0$ for all $i>0$, and therefore $\Ext_{R}^i(N, N\oplus R)=0$ for all $i>0$, as $\widehat R$ is a pure extension of $R$.
Our hypothesis yields that $N$ is a free $R$-module, so that $M=\widehat{N}$ is a free $\widehat{R}$-module, as desired.
\end{proof}

\begin{ac}
The authors thank Srikanth Iyengar for his valuable comments.
\end{ac}



\end{document}